\documentclass{fundam}
\usepackage{url} 
\usepackage[ruled,lined]{algorithm2e}
\usepackage{graphicx}
\usepackage{tikz}
\usetikzlibrary{automata, positioning, arrows}

\usepackage{enumerate} 

\begin{document}

%
\setcounter{page}{1}
\publyear{2021}
\papernumber{0001}
\volume{XXX}
\issue{X}
%

\title{Construction of networks by associating with submanifolds of almost Hermitian manifolds}

\address{Department of Mathematics, Science Faculty, Ege University, Bornova, 35100, Izmir, Turkey}

\author{Arif Gursoy\thanks{Corresponding Author}\\
Department of Mathematics \\
Ege University\\ Bornova, 35100 Izmir, Turkey\\
arif.gursoy{@}ege.edu.tr}

\maketitle

\runninghead{A. Gursoy}{Construction of networks by associating with submanifolds of almost Hermitian manifolds}

\begin{abstract}
The idea that data lies on a non-linear space has brought up the concept of manifold learning \cite{ZX} as a part of machine learning and such notion is one of the most important research field of today. 
\end{abstract}

\begin{keywords}
Graph theory, digraphs, almost Hermitian manifold, submanifolds, manifold learning, machine learning, artificial intelligence
\end{keywords}

\section{Introduction}
Graph theory can be used to model computer networks, social networks, communications networks, information networks, software design, transportation networks, biological networks, etc. So this theory is applicable in many real-world mathematical modeling. Therefore, this theory is the most active areas of mathematical research.

On the other hand, one of the most active research area of differential geometry is the submanifold theory of complex manifolds. A submanifold of an almost Hermitian manifold is characterized by the behavior of tangent space of the submanifold of almost Hermitian manifold under the complex structure of the ambient manifold. In this way, we have various submanifolds titled as holomorphic, totally real, CR, slant, semi slant, hemi-slant, bi-slant for almost Hermitian manifolds. In fact, the theory of submanifolds of almost Hermitian manifolds is still main active area of complex differential geometry, see:\cite{ALVY, Lee2015, Vilcu2018} for recent results.

Manifold learning method is one of the most exciting developments in machine learning recently. Manifold learning has been applied in utilizing semi-supervised learning \cite{ZX}. Moreover, Vakulenko and Radulescu have used the theory of invariant and inertial manifold to prove the realization of prescribed dynamics by networks in patterning by centralized gene networks \cite{Vakulenko}. Furthermore, manifolds also play an important role in public health. Fiorini has defined the Riemannian manifold, which is isomorphic to traditional information geometry Riemannian manifold, for noise reduction in theoretical computerized tomography providing many competitive computational advantages over the traditional Euclidean approach \cite{Fiorini}. Besides, Monti et. al. have introduced a general framework, geometric deep learning, enabling to design of convolutional deep architectures on manifolds and graphs \cite{MBMRSB}.

Also, Carriazo and Fernandez \cite{carriazo2004submanifolds} have constructed a relation between slant surface and graph theory. Later, they have related graph theory with vector spaces of even dimension \cite{carriazo2009submanifolds, boza2012graphs}. Their work was restricted to slant submanifolds. We believe that further use of graph theory is possible in the theory of submanifolds.

By considering vast literature of graph theory and submanifold theory, one expects more relations between these research areas. In this direction, the aim of this paper is to examine the relation among various submanifolds of almost Hermitian manifolds by using graph theory. We note that our approach is different from the approach considered in \cite{carriazo2004submanifolds} and \cite{carriazo2009submanifolds}. They only considered adapted frame of slant surface and they used them to characterize CR-submanifolds by means of trees. Later they have extended this approach for weakly associated graph. In this paper, we give relations between submanifolds of Hermitian manifolds in terms of graph theory notions.

\section{Preliminaries}

In this section, we are going to recall certain notions used graph theory to be used in this paper from \cite{bang2008digraphs, belmonte2019new, bondy1976graph, chartrand2010graphs, cormen2009introduction, rosen2013discrete}.
For those who are not familiar with the theory of graphs (especially for readers working with the submanifolds theory), we specifically recall the basic definitions from graph theory.

A graph $G = (V, E)$ consists of a nonempty set $V$ of vertices and a set $E$ of edges. Each edge has either one or two vertices connected with it, called its endpoints. An edge connects its endpoints. Two distinct vertices $u, v$ in a graph G are called adjacent (or neighbors) in $G$ if there is an edge e between $u$ and $v$, where the edge $e$ is called incident with the vertices $u$ and $v$ and $e$ connects $u$ and $v$. The set of all neighbors of a vertex $v$ of $G = (V, E)$ is denoted by $N(v)$. If $A \subset V$, we denote by $N(A)$ the set of all vertices in $G$ that are adjacent to at least one vertex in $A$. The degree of a vertex in a graph is the number of edges incident with it. The degree of the vertex $v$ is denoted by $d(v)$ and $d(v) = \vert N(v)\vert$. The graph theory can be divided into two branches as undirected and directed graphs. \cite{rosen2013discrete}

A directed graph (digraph) $D$ is a finite nonempty set of objects called vertices together with a set of ordered pairs of distinct vertices of $D$ called directed edges or arcs. For a digraph $D=(V, A)$, the vertex set of $D$ is denoted by $V(D)$ or simply $V$ and the arc set of $D$ is denoted by $A(D)$ or $A$. Each arc is an ordered pair of vertices. The arc $(u,v)$ is said to start at $u$ and end at $v$. The in-degree of a vertex v, $d^-(v)$, is the number of edges which end at $v$. The out-degree of $v$, $d^+(v)$, is the number of edges with $v$ as their initial vertex. Also, for a vertex $v \in V(D)$, $N^{-}_D(v)$ and $N^{+}_D(v)$ are respectively called out-neighbors and in-neighbors where $N^{-}_D(v)=\{u \vert (u,v) \in A(D), u \in V(D)\}$ and $N^{+}_D(v)=\{u | (v,u) \in A(D), u \in V(D)\}$.  \cite{chartrand2010graphs, bang2008digraphs, sedgewick2015algorithms, rosen2013discrete}

In a digraph $D = (V, A)$, given a pair of vertices $u$ and $v$, whether or not there is a path from $u$ to $v$ in the digraph is useful to know. The transitive closure of $D$ is to construct a new digraph, $D^* = (V, A^*)$, such that there is an arc $(u, v)$ in $D^*$ if and only if there is a path from $u$ to $v$ in $D$. \cite{cormen2009introduction}

A walk $W = x_1 a_1 x_2 a_2 x_3 ... x_{k-1} a_{k-1}x_k$ is a sequence of vertices $x_i$ and arcs $a_j$ in $D$ such that the tail and head of $a_i$ is $x_i$ and $x_{i+1}$ for every $i \in [k - 1]$, respectively. The set of vertices and arcs of the walk $W$ are denoted $V(W)$ and $A(W)$, respectively. $W$ is denoted without arcs as $x_1x_2...x_k$ and shortly $(x_1,x_k)$-walk. If $x_1 = x_k$ then $W$ is a closed walk, and otherwise $w$ is an open walk. If $W$ is an open walk, the vertices $x_1$ and $x_k$ are end-vertices and named as the initial and the terminal vertex of $W$, respectively. The length of a walk is the number of its arcs and the walk $W$ above has length $k-1$. \cite{bang2008digraphs}

A trail is a walk in which all arcs are distinct. $W$ is called a path, if the vertices of a trail $V(W) \subset V(D)$ are distinct. If the vertices $x_1, x_2, ..., x_{k-1}$ are distinct, $k \geq 3$ and $x_1 = x_k$, then $W$ is a cycle. The longest path in $D$ is a path of maximum length in $D$. \cite{bang2008digraphs}

\begin{proposition}
\label{prop1} \cite{bang2008digraphs} Let $D$ be a digraph and let $ x, y$ be a pair of distinct vertices in $D$. If $D$ has an $(x, y)$-walk $W$, then $D$ contains an $(x, y)$-path $P$ such that $A(P) \subseteq A(W)$. If $D$ has a closed $(x, x)-$walk $W$, then $D$ contains a cycle $C$ through $x$ such that $A(C) \subseteq A(W)$.
\end{proposition}

An oriented graph is a digraph with no cycle of length two \cite{bang2008digraphs}. For a digraph $D$, the Underlying Graph of $D$ is the undirected graph engendered utilizing all vertices in $V(D)$, and superseding all of the arcs in $A(D)$ with undirected edges. \cite{bondy1976graph}

If a digraph $D$ has an $(x,y)$-walk, then the vertex $y$ is reachable from the vertex $x$. Every vertex is reachable from itself specifically. By Proposition \ref{prop1}, $y$ is reachable from $x$ if and only if $D$ contains an $(x, y)$-path. If every pair of vertices in digraph $D$ is mutually reachable then $D$ is strongly connected (or shortly strong). A strong component of digraph $D$ is a maximal induced strong subdigraph in $D$. If $D_1, ... ,D_t$ are the strong components of $D$, then precisely $V(D_1) \cup ... \cup V(D_t) = V(D)$. If a digraph $D$ is not strongly connected and if the underlying graph of $D$ is connected, then $D$ is said to be weakly connected. \cite{bang2008digraphs, sedgewick2015algorithms}

Pseudograph is a graph having parallel edges and loops, and multigraph is a pseudograph with no loops. If every pair of distinct vertices are adjacent in a multigraph then the multigraph is complete.

A multigraph $H$ is called as $p-$partite if there is a partition into p sets $V(H) = V_1 \cup V_2 \cup ... \cup V_p$ where $V_i \cap V_j = \O$ for every $i \neq j$. In particular, when $p = 2$ the graph is called a bipartite graph. A bipartite graph $B$ is denoted by $B = (V_1, V_2;E)$. If the edge $(x,y)$ is in $p-$partite multigraph $H$ where all $x \in V_i$, $y \in V_j$ for $i \neq j$ then $H$ is complete $p-$partite. \cite{bang2008digraphs}

A digraph $D = (V,A)$ is symmetric if arc $(x,y) \in A$ implies arc $(y,x) \in A$. A matching $M$ is an arc set having no common end-vertices and loops in $D$. Also, the arcs of $M$ are independent if $M$ is a matching. If a matching $M$ implicates the highest number of arcs in $D$, then $M$ is maximum. Besides, a maximum matching is perfect if it has $\frac{|A(D)|}{2}$ arcs. A set $Q$ of vertices in a directed pseudograph $H$ is independent if there are no arcs between vertices in $Q$. The independence number of $H$ is the size of the independent set having maximum cardinality in $H$. A coloring of a digraph $H$ is a partition of $V(H)$ into disjoint independent sets. The minimum number of independent sets in the coloring of $H$ is the chromatic number of $H$. A simple directed graph is a digraph that has no multiple arcs or loops. if a digraph contains no cycle, then it is acyclic and called acyclic digraph. \cite{bang2008digraphs}

The eccentricity $e(v)$ of a vertex $v$ is the distance from $v$ to the farthest vertex from itself. The radius ($rad$) of $D $ is the minimum eccentricity, and the diameter ($diam$) is the maximum eccentricity. Besides, a vertex $v$ is central if $e(v) = rad(D)$, and $v$ is peripheral if $e(v) = diam(D)$. \cite{chartrand1997distance}

Let $D=(V,A)$ be a digraph, $V(D)=n$ and $S$ $\subset$ $V(D)$. $S$ is a dominating set of $D$ if each vertex $v \in V(D) - S$ is dominated by at least a vertex in $S$. A dominating set of $D$ having the smallest cardinality is called the minimum dominating set of $D$. Also, the cardinality of the minimum dominating set is called the domination number of $D$ \cite{lee1998domination, pang2010dominating}

Let $r$ be a root vertex in $D$. A directed spanning tree $T$ starting from $r$ is a subdigraph of $D$ such that the undirected form of $T$ is a tree and there is a directed unique $(r,v)$-path in $T$ for each $v \in V(T)-r$. \cite{bang2008digraphs}

The vertex-integrity of a digraph $D$ is defined by $I(D) = min\{|F| + m(D - F): F\subseteq V(D)\}$, where $m(D - F)$ indicates the maximum order of a strong component of $D - F$. If $I(D) = |F| + m(D - F)$ then $F$ is called as an $I$-set of $D$. In addition, the arc-integrity of a digraph $D$, shortly $I^{'}(D)$, is described as the minimum value of $\{|F| + m(D - F): F\subseteq A(D)\}$. The set $F$ is called as an $I^{'}$-set of $D$ if $I^{'}(D) = |F| + m(D - F)$. \cite{Vandell1996}

\begin{proposition}\label{prop.integrity} \cite{Vandell1996}
If $S$ is a subdigraph of $D$ then $I(S) \leq I(D)$ and $I^{'}(S) \leq I^{'}(D)$.
\end{proposition}

\section{Construction of digraphs by relations among submanifolds of almost Hermitian manifolds}

Let $( M,g)$ be a Riemannian manifold. $(M,g)$ is called an almost Hermitian manifold if there is a (1,1) tensor field on $M$ such that $J^2=-I$, where $I$ is the identity map on the tangent bundle of $M$, and $g(JX,JY)=g(X,Y)$ for vector fields $X,Y$ on $M$. Moreover if $J$ is parallel with respect to any vector field $X$, then $(M,J,g)$ is called a Kaehler manifold \cite{yano1985manifolds}. There are various submanifolds of an almost Hermitian manifold based on the behavior of the tangent space of the submanifold at a point under the almost complex structure $ J$. Let $N$ be a submanifold of an almost Hermitian manifold and $T_pN$ the tangent space at a point p belongs to $N$. Then if $T_pN$ is invariant with respect to $J_p$ for any point $p$, then $N$ is called holomorphic (or complex) submanifold \cite{yano1985manifolds}. We denote the normal space at $p$ by $T_pN^\perp$. A submanifold of an almost Hermitian manifold is called an anti-invariant submanifold if $JT_p N \subseteq T_pN^\perp$ \cite{yano1985manifolds}. As a generalization of holomorphic submanifold and anti-invariant submanifolds, a submanifold $M$ of a Kaehler manifold $N$ is called CR-submanifold \cite{bejancu2012geometry} if there are two orthogonal complementary distributions $\mathcal{D}_1$ and $\mathcal{D}_2$ such that $\mathcal{D}_1$ is invariant with respect to $J$ and $\mathcal{D}_2$ is anti-invariant with respect to $J$ for every point $p \in M$. It is clear that if $\mathcal{D}_1=\{0\}$, then a CR-submanifold becomes an anti-invariant submanifold. If $\mathcal{D}_2=\{0\}$, then $M$ becomes a holomorphic submanifold. Another generalization of holomorphic submanifolds and anti-anti-invariant submanifolds is slant submanifolds. Let $N$ be a submanifold of an almost Hermitian manifold $M$. The submanifold $N$ is called slant \cite{chen1990geometry} if for each non-zero vector $X$ tangent to $N$ the angle $\theta(X)$ between $JX$ and $T_pN$ is a constant, i.e, it does not depend on the choice of $p \in M$ and $X \in T_p N$. $\theta$ is called the slant angle. It is clear that if $\theta(X)=0$ then $N$ becomes a holomorphic submanifold. If $\theta(X) = \pi / 2$, $N$ becomes an anti-invariant submanifold. We will use the $v_1$, $v_2$, $v_3$, and $v_4$ to represent the submanifolds holomorphic, CR, anti-invariant and slant, respectively.

Digraph $D_1 = (V,A)$ has four vertices, $V(D_1) = \{v_1,v_2,v_3,v_4\}$, and four arcs, $A(D_1)=\{(v_2,v_1),(v_2,v_3),(v_4,v_1),(v_4,v_3)\}$ in Fig. \ref{fig1}. $D_1$ has the maximum length of one as the longest path. $D_1$ has 2 vertices ($v_2$ and $v_4$) which are not reachable. Topological sort of $D_1$ is $v_4-v_2-v_3-v_1$. $rad(D_1)=1$, the radius of $D_1$ is $v_2 \to v_1$. $diam(D_1)=1$, the diameter of $D_1$ is the same as the radius. Also, in $D_1$, there is no center vertex, but two peripheral vertices such as $v_2$ and $v_4$.

    \begin{figure}
    \centering
\begin{tikzpicture}[->,>=stealth',shorten >=1pt,thick]
\tikzset{vertex/.style = {shape=circle,draw,minimum size=1.5em}}
\tikzset{edge/.style = {->,> = latex'}}
\node[vertex] (1) at  (0,0) {$v_1$};
\node[vertex] (2) at  (3,0) {$v_2$};
\node[vertex] (3) at  (3,-3) {$v_3$};
\node[vertex] (4) at  (0,-3) {$v_4$};
\draw[edge] (2) to (1);
\draw[edge] (2) to (3);
\draw[edge] (4) to (1);
\draw[edge] (4) to (3);
\end{tikzpicture}
    \caption{Digraph $D_1$ built by submanifolds holomorphic, CR, anti-invariant and slant}
    \label{fig1}
    \end{figure}
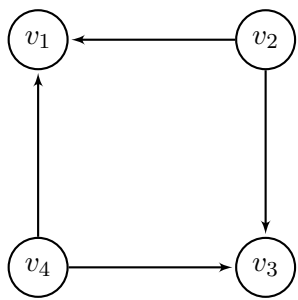

\begin{theorem}
\label{teo31}
For a digraph $D_1$ constructed by the four submanifolds holomorphic, CR, anti-invariant and slant considering as the vertices $v_1$, $v_2$, $v_3$, and $v_4$, respectively,

\begin{enumerate}[i]
\item $D_1$ is a bipartite digraph as well as a complete bipartite digraph.
\item $D_1$ has a perfect matching.
\item The independence number of $D_1$ is 2.
\item The chromatic number of $D_1$ is 2.
\item $D_1$ has no directed spanning tree.
\item The domination number of $D_1$ is 2.

\end{enumerate}

\end{theorem}

\begin{proof}
\begin{enumerate}[i]
\item There exists a partition $V_1$ and $V_2$ of $V(D_1)$ into two partite sets for the submanifolds in $D_1$: $V_1=\{v_1,v_3\}$ and $V_2=\{v_2,v_4\}$. Owing to $V(D_1) = V_1 \cup V_2$ and $V_1 \cap V_2 = \O$, then $D_1$ is a bipartite digraph. \\
Besides, for every submanifold, $x \in V_1, y \in V_2$, a connection from x to y (i.e. an arc $(x,y)$) is in $D_1$. Therefore, $D_1$ is a complete bipartite digraph.

\item There is a matching $M=\{(v_2,v_1),(v_4,v_3)\} \subset A(D_1)$ in $D_1$. Each element (arc or connection between two submanifolds) in $M$ is independent, i.e. no common vertices, and $M$ is maximum. Also, $M$ is perfect so that $|M|=\frac{|A(D_1)|}{2}$. It is obvious that $D_1$ has a perfect matching.

\item The subset $\widetilde{V}=\{v_2, v_4\} \subset V(D_1)$ is one of the independent sets having maximum cardinality and the size of maximum independent submanifolds set is two. This also means that there is no relation between submanifolds $v_2$ and $v_4$. Then, the independence number of $D_1$ is two.

\item $V_1=\{v_2, v_4\}$ and $V_2=\{v_1, v_3\}$ are two subsets of $V(D_1)$. $V_i (i=1,2)$ are all independent sets providing the minimum number of cardinality at the same time. Hence, the minimum number of independent sets of $D_1$ is two. Then, the chromatic number of $D_1$ is two.

\item There is no root vertex where a subdigraph T of $D_1$ contains a directed path from the root to any other vertex in $V(D_1)$. Then, $D_1$ has no directed spanning tree.

\item There is a subset $\widetilde{V}=\{v_2, v_4\} \subset V(D_1)$ that including minimum cardinality of vertices in $D_1$. Considering this subset, for each vertex $v \in \widetilde{V}$ and $u \in V(D_1)-\widetilde{V}$, (v,u) is an arc in $D_1$. The domination number is two, because of no smaller cardinality of dominating sets in $D_1$.
\end{enumerate}
\end{proof}

\begin{corollary}\label{cor31}
In the submanifold network represented by $D_1$ in Fig. \ref{fig1}, the submanifolds, CR $(v_2)$ and slant $(v_4)$, cannot be derived by the other submanifolds, because the in-degree of these vertices (submanifolds) are zero in $D_1$, $d^{-}(v_2) = d^{-}(v_4) = 0$.
In addition, whereas CR and slant subamnifolds cannot be mutually derived as between holomorphic $(v_1)$ and anti-invariant $(v_3)$, holomorphic and anti-invariant submanifolds can be derived separately from CR and slant from $N^{-}_{D_1}(v_1) = N^{-}_{D_1}(v_3) = \{v_2, v_4 \}$.
\end{corollary}

We now recall the notion of hemi-slant submanifolds of an almost Hermitian manifold. Let $M$ be an almost Hermitian manifold and $N$ a real submanifold of $M$. Then we say that $N$ is a hemi-slant submanifold \cite{carriazoalfonso2000,sahin2009warped} if there exist two orthogonal distributions $\mathcal{D}^\perp$ and $\mathcal{D}^\theta$ on $N$ such that
\begin{enumerate}
\item $TN$ admits the orthogonal direct decomposition $TN = \mathcal{D}^\perp \oplus \mathcal{D}^\theta$.
\item The distribution $\mathcal{D}^\perp$ is an anti-invariant distribution, i.e., $J\mathcal{D}^\perp \subset TM^\perp$.
\item The distribution $\mathcal{D}^\theta$ is slant with slant angle $\theta$.
\end{enumerate}
It is easy to see that if $\mathcal{D}^\perp=\{0\}$, $N$ becomes a slant submanifold with a slant angle $\theta$. If $\mathcal{D}^\theta=\{0\}$, then $N$ becomes an anti-invariant submanifold. Moreover if $\theta=0$, then $N$ becomes a CR-submanifold. Furthermore, if $\mathcal{D}^\perp=\{0\}$ and $\theta=0$, then $N$ becomes a holomorphic submanifold. We denote hemi-slant submanifolds by $v_6$.

Digraph $D_2 = (V,A)$ is an extension of $D_1$, and has five vertices, $V(D_2) = \{v_1,v_2,v_3,v_4,v_6\}$, and seven arcs, $A(D_2)=\{(v_2,v_1), (v_2,v_3), (v_4,v_1), (v_4,v_3), (v_6,v_1), (v_6,v_2), (v_6,v_3)\}$ in Fig. \ref{fig2}. $D_2$ has the maximum length of two as the longest path. It has 2 vertices ($v_4$ and $v_6$) which are not reachable. Topological sort of $D_2$ is $v_6-v_4-v_2-v_3-v_1$. $rad(D_2)=1$, the radius of $D_2$ is $v_2 \to v_1$. $diam(D_2)=1$, the diameter of $D_2$ is the same as the radius. Also, there is no center vertex but three peripheral vertices such as $v_2$, $v_4$ and $v_6$.

    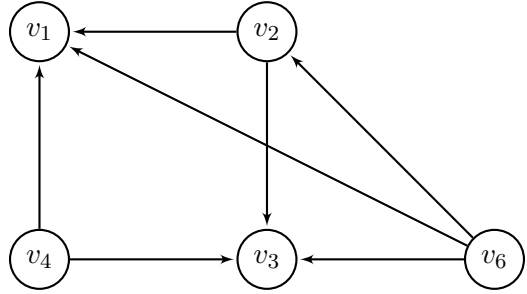
\begin{figure}
    \centering
\begin{tikzpicture}[->,>=stealth',shorten >=1pt,thick]
\tikzset{vertex/.style = {shape=circle,draw,minimum size=1.5em}}
\tikzset{edge/.style = {->,> = latex'}}
\node[vertex] (1) at  (0,0) {$v_1$};
\node[vertex] (2) at  (3,0) {$v_2$};
\node[vertex] (3) at  (3,-3) {$v_3$};
\node[vertex] (4) at  (0,-3) {$v_4$};
\node[vertex] (6) at  (6,-3) {$v_6$};
\draw[edge] (2) to (1);
\draw[edge] (2) to (3);
\draw[edge] (4) to (1);
\draw[edge] (4) to (3);
\draw[edge] (6) to (1);
\draw[edge] (6) to (2);
\draw[edge] (6) to (3);
\end{tikzpicture}
    \caption{Digraph $D_2$ built by submanifolds in $D_1$ and the hemi-slant submanifold}
    \label{fig2}
    \end{figure}

\begin{theorem}
\label{teo32}
For the digraph $D_2$ created by adding the hemi-slant submanifolds as vertex $v_6$ to the $D_1$,
\begin{enumerate}[i]
\item $D_2$ is a three-partite digraph.
\item The maximum matching is 2.
\item The independence number is 2.
\item The chromatic number is 3.
\item $D_2$ has no directed spanning tree.
\item The domination number is 2.
\end{enumerate}
\end{theorem}

\begin{proof}
\begin{enumerate}[i]
\item There exists a partition $V_1=\{v_1, v_3\}$, $V_2=\{v_2\}$ and $V_3=\{v_4,v_6\}$ of $V(D_2)$. These three subsets are three partite sets because of following attributes: $V(D_2) = \bigcup\limits_{i=1}^{3} V_{i}$ and $V_i \cap V_j = \O$ ($i,j=1,2,3$ and $i \neq j$).
Then, $D_2$ is a three-partite digraph.

\item There is an arc subset $M=\{(v_6,v_1),(v_4,v_3)\}$ in $D_2$, and $|M| = 2$. In $M$, there is no common vertices and loops, that is $M$ is a matching. Also, there is no arc subset having greater cardinality than $M$. Therefore, $M$ is maximum matching in $D_2$.

\item The maximum independent set and independence number of $D_2$ is the same as $D_1$. See Theorem \ref{teo31}-iii.

\item The minimum number of disjoint independent sets of $D_2$ is three: $V_1=\{v_1, v_3\}$, $V_2=\{v_2\}$ and $V_3=\{v_4,v_6\}$. Then, chromatic number of $D_2$ is three.

\item No root vertex that contains a directed path from the root to any other vertex in $V(D_2)$. Then, $D_2$ has no directed spanning tree.

\item There is a subset $\widetilde{V}=\{v_4, v_6\} \subset V(D_2)$. Considering this subset, that including the minimum cardinality of vertices in $D_2$ as a dominating set, for each vertex $v \in \widetilde{V}$ and $u \in V(D_2)-\widetilde{V}$, (v,u) is an arc in $D_2$. Clearly, the domination number is two.

\end{enumerate}
\end{proof}

\begin{corollary}\label{cor32}
In the submanifold network represented by $D_2$ in Fig. \ref{fig2}, the submanifolds, slant $(v_4)$ and hemi-slant $(v_6)$, cannot be derived by the other submanifolds, because $d^{-}(v_4) = d^{-}(v_6) = 0$ in $D_2$.
Also, holomorphic $(v_1)$ and anti-invariant $(v_3)$ submanifolds can be derived separately by CR $(v_2)$, slant and hemi-slant since $N^{-}_{D_2}(v_1)=N^{-}_{D_2}(v_3)=\{ v_2, v_4, v_6 \}$.
\end{corollary}

We now recall the notion of semi-slant submanifolds of an almost Hermitian manifold. Let $M$ be an almost Hermitian manifold and $N$ a real submanifold of $M$. Then we say that $N$ is a semi-slant submanifold \cite{papaghiuc1994} if there exist two orthogonal distributions $\mathcal{D}$ and $\mathcal{D}^\theta$ on $N$ such that
\begin{enumerate}
\item $TN$ admits the orthogonal direct decomposition $TN = \mathcal{D} \oplus \mathcal{D}^\theta$.
\item The distribution $\mathcal{D}$ is an invariant distribution, i.e., $J(\mathcal{D})=\mathcal{D}$.
\item The distribution $\mathcal{D}^\theta$ is slant with slant angle $\theta$.
\end{enumerate}
It is easy to see that if $\mathcal{D}=\{0\}$, $ N$ becomes a slant submanifold with a slant angle $\theta$. If $\mathcal{D}^\theta=\{0\}$, then $ N$ becomes a holomorphic submanifold. Moreover if $\theta=\frac{\pi}{2}$, then $ N$ becomes a CR-submanifold. Furthermore, if $\mathcal{D}=\{0\}$ and $\theta=\frac{\pi}{2}$, then $ N$ becomes an anti-invariant submanifold. We denote semi-slant submanifolds by $v_5$.

Digraph $D_3 = (V,A)$ is another extension of $D_1$, and has five vertices, $V(D_3) = \{v_1,v_2,v_3,v_4,v_5\}$, and seven arcs, $A(D_3)=\{(v_2,v_1),(v_2,v_3),(v_4,v_1),(v_4,v_3),(v_5,v_2),(v_5,v_3),(v_5,v_4)\}$ in Fig. \ref{fig3}. $D_3$ has the maximum length of two as the longest path. It has a vertex ($v_5$) which is not reachable. Using transitive closure, $D_3$ has only one new direct connection such as $v_5 \to v_1$. Topological sort of $D_3$ is $v_5-v_4-v_2-v_3-v_1$. $rad(D_3)=1$, the radius of $D_3$ is $v_2 \to v_1$. $diam(D_3)=2$, the diameter of $D_3$ is $v_5 \to v_2 \to v_1$. Also, in $D_3$, there are two center vertices as $v_2$ and $v_4$, and one peripheral vertex as $v_5$.

    \begin{figure}
    \centering
\begin{tikzpicture}[->,>=stealth',shorten >=1pt,thick]
\tikzset{vertex/.style = {shape=circle,draw,minimum size=1.5em}}
\tikzset{edge/.style = {->,> = latex'}}
\node[vertex] (1) at  (0,0) {$v_1$};
\node[vertex] (2) at  (3,0) {$v_2$};
\node[vertex] (3) at  (3,-3) {$v_3$};
\node[vertex] (4) at  (0,-3) {$v_4$};
\node[vertex] (5) at  (6,0) {$v_5$};
\draw[edge] (2) to (1);
\draw[edge] (2) to (3);
\draw[edge] (4) to (1);
\draw[edge] (4) to (3);
\draw[edge] (5) to (2);
\draw[edge] (5) to (3);
\draw[edge] (5) to (4);
\end{tikzpicture}
    \caption{Digraph $D_3$ built by submanifolds in $D_1$ and the semi-slant submanifold}
    \label{fig3}
    \end{figure}
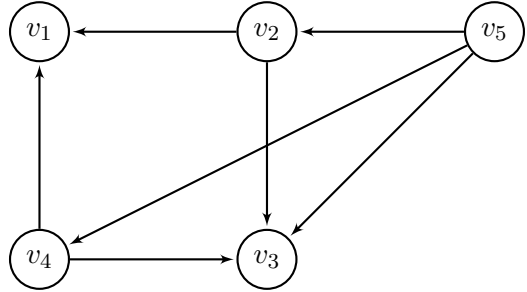

\begin{theorem}
\label{teo33}
For the digraph $D_3$ created by adding the semi-slant submanifolds as vertex $v_5$ to the $D_1$,

\begin{enumerate}[i]
\item $D_3$ is a three-partite digraph.
\item The maximum matching is 2.
\item The independence number is 2.
\item The chromatic number is 3.
\item $D_3$ has a directed spanning tree.
\item The domination number is 2.
\end{enumerate}
\end{theorem}

\begin{proof}
\begin{enumerate}[i]
\item There exists a partition $V_1=\{v_1,v_3\}$, $V_2=\{v_2, v_4\}$ and $V_3=\{v_5\}$ of $V(D_3)$ as three partite sets in $D_3$, and the subsets provide following properties: $V(D_3) = \bigcup\limits_{i=1}^{3} V_{i}$ and $V_i \cap V_j = \O$ ($i,j=1,2,3$ and $i \neq j$). In that case, $D_3$ is a three-partite digraph.

\item There is an arc subset $M=\{(v_2,v_1),(v_4,v_3)\}$ in $D_3$, and $|M| = 2$. Because of no common vertices and no loops in $M$, $M$ is a matching. Furthermore, $M$ has the maximum cardinality so that $M$ is the maximum matching in $D_3$.

\item The maximum independent set and independence number of $D_3$ is the same as $D_1$. See Theorem \ref{teo31}-iii.

\item The minimum number of disjoint independent sets of $D_3$ is three: $V_1=\{v_1, v_3\}$, $V_2=\{v_2, v_4\}$ and $V_3=\{v_5\}$. It follows that the chromatic number of $D_3$ is three.

\item $D_3$ has a unique directed spanning tree of length 4 and rooted at $v_5$ such as in Fig. \ref{fig4}. It also means that there is a transformation from submanifolds $v_5$ to all other submanifolds in $D_3$.

    \begin{figure}[h]
    \centering
\begin{tikzpicture}[->,>=stealth',shorten >=1pt,thick]
\tikzset{vertex/.style = {shape=circle,draw,minimum size=1.5em}}
\tikzset{edge/.style = {->,> = latex'}}
\node[vertex] (1) at  (0,0) {$v_1$};
\node[vertex] (2) at  (3,0) {$v_2$};
\node[vertex] (3) at  (3,-3) {$v_3$};
\node[vertex] (4) at  (0,-3) {$v_4$};
\node[vertex] (5) at  (6,0) {$v_5$};
\draw[edge] (2) to (1);
\draw[edge] (5) to (2);
\draw[edge] (5) to (3);
\draw[edge] (5) to (4);
\end{tikzpicture}
    \caption{Directed spanning tree in $D_3$}
    \label{fig4}
    \end{figure}
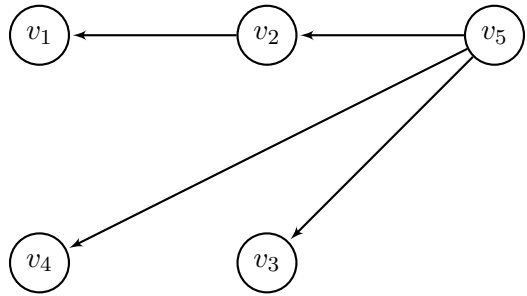

\item There is a subset $\widetilde{V}=\{v_4, v_5\} \subset V(D_3)$. According to this subset, that having the minimum cardinality, and for each vertex $v \in \widetilde{V}$ and $u \in V(D_3)-\widetilde{V}$, (v,u) is an arc in $D_3$, the domination number is two.
\end{enumerate}
\end{proof}

\begin{corollary}\label{cor33}
In the submanifold network represented by $D_3$ in Fig. \ref{fig3}, while no submanifolds can be transformed to semi-slant $(v_5)$ submanifold since $N^{-}_{D_3}(v_5)=\emptyset$, all other submanifolds (holomorphic $(v_1)$, CR $(v_2)$, anti-invariant $(v_3)$ and slant $(v_4)$) can be obtained from semi-slant submanifold because of existence of a directed spanning tree with a root vertex $v_5$ (Fig. \ref{fig4}).
\end{corollary}

Digraph $D_4 = (V,A)$ has six vertices, $V(D_4) = \{v_1,v_2,v_3,v_4,v_5,v_6\}$, and 10 arcs, $A(D_4)=\{(v_2,v_1),(v_2,v_3),(v_4,v_1),(v_4,v_3),(v_5,v_2),(v_5,v_3),(v_5,v_4),(v_6,v_1),(v_6,v_2),(v_6,v_3)\}$ in Fig. \ref{fig5}. $D_4$ has the maximum length of two as the longest path. It has 2 vertices ($v_5$ and $v_6$) which are not reachable. Using transitive closure, $D_4$ has only one new direct connection such as $v_5 \to v_1$. The topological sort of $D_4$ is $v_6-v_5-v_4-v_2-v_3-v_1$. $rad(D_4)=1$, the radius of $D_4$ is $v_2 \to v_1$. $diam(D_4)=2$, the diameter of $D_4$ is $v_5 \to v_2 \to v_1$. Also, in $D_4$, there are three center vertices as $v_2$, $v_4$ and $v_6$, and one peripheral vertex as $v_5$.

    \begin{figure}
    \centering
\begin{tikzpicture}[->,>=stealth',shorten >=1pt,thick]
\tikzset{vertex/.style = {shape=circle,draw,minimum size=1.5em}}
\tikzset{edge/.style = {->,> = latex'}}
\node[vertex] (1) at  (0,0) {$v_1$};
\node[vertex] (2) at  (3,0) {$v_2$};
\node[vertex] (3) at  (3,-3) {$v_3$};
\node[vertex] (4) at  (0,-3) {$v_4$};
\node[vertex] (5) at  (6,0) {$v_5$};
\node[vertex] (6) at  (6,-3) {$v_6$};
\draw[edge] (2) to (1);
\draw[edge] (2) to (3);
\draw[edge] (4) to (1);
\draw[edge] (4) to (3);
\draw[edge] (5) to (2);
\draw[edge] (5) to (3);
\draw[edge] (5) to (4);
\draw[edge] (6) to (1);
\draw[edge] (6) to (2);
\draw[edge] (6) to (3);
\end{tikzpicture}

    \caption{Digraph $D_4$ built by submanifolds in $D_3$ and the hemi-slant submanifold}
    \label{fig5}
    \end{figure}
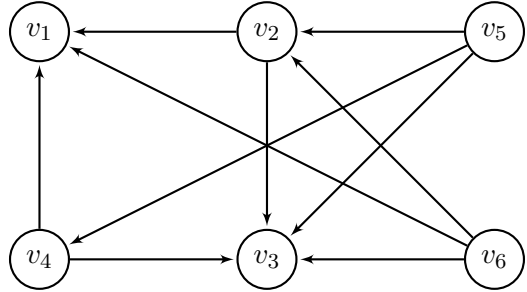

\begin{theorem}
\label{teo34}
For the digraph $D_4$ created by adding the hemi-slant submanifolds as vertex $v_6$ to the $D_3$,

\begin{enumerate}[i]
\item $D_4$ is a three-partite digraph.
\item $D_4$ has a perfect matching.
\item The independence number is 2.
\item The chromatic number is 3.
\item $D_4$ has no directed spanning tree.
\item The domination number is 2.
\end{enumerate}
\end{theorem}

\begin{proof}
\begin{enumerate}[i]
\item There exists a partition $V_1=\{v_1,v_3\}$, $V_2=\{v_2, v_4\}$ and $V_3=\{v_5, v_6\}$ of $V(D_4)$ as three subsets, and these subsets provide that $V(D_4) = \bigcup\limits_{i=1}^{3} V_{i}$ and $V_i \cap V_j = \O$ ($i,j=1,2,3$ and $i \neq j$). Under these conditions, $D_4$ is a three-partite digraph.

\item There is an arc subset $M=\{(v_2,v_1),(v_5,v_4),(v_6,v_3)\}$ in $D_4$, and $|M| = 3$. On conditions that no common vertices and no loops in $M$ and $|M|=\frac{|A(D_4)|}{2}$, $M$ is perfect matching that's why $D_4$ has a matching also perfect.

\item The maximum independent set and the independence number of $D_4$ is the same as $D_1$. See Theorem \ref{teo31}-iii.

\item The minimum number of disjoint independent sets of $D_4$ is three: $V_1=\{v_1, v_3\}$, $V_2=\{v_2, v_4\}$ and $V_3=\{v_5, v_6\}$. Then, the chromatic number of $D_4$ is three.

\item No root vertex that contains a directed path from the root to any other vertex in $V(D_4)$. Then, $D_4$ has no directed spanning tree.

\item There is a subset $\widetilde{V}=\{v_5, v_6\} \subset V(D_4)$. According to this subset, that having the minimum cardinality, and for each vertex $v \in \widetilde{V}$ and $u \in V(D_4)-\widetilde{V}$, (v,u) is an arc in $D_4$ so that the domination number is two.
\end{enumerate}
\end{proof}

\begin{corollary}\label{cor34}
In the submanifold network represented by $D_4$ in Fig. \ref{fig5}, semi-slant $(v_5)$ and hemi-slant $(v_6)$ submanifolds cannot be obtained by any other submanifolds because $d^{-}(v_5) = d^{-}(v_6) = 0$. Besides, no submanifolds can be derived from holomorphic $(v_1)$ and anti-invariant $(v_3)$ submanifolds since $N^{-}_{D_4}(v_1) = N^{-}_{D_4}(v_3) = \emptyset$.
\end{corollary}

We now recall the notion of bi-slant submanifolds of an almost Hermitian manifold. Let $M$ be an almost Hermitian manifold and $ N$ a real submanifold of $M$. Then we say that $ N$ is a bi-slant submanifold \cite{carriazoalfonso2000} if there exist two orthogonal distributions $\mathcal{D}^{\theta_1}$ and $\mathcal{D}^{\theta_2}$ on $ N$ such that
\begin{enumerate}
\item $TN$ admits the orthogonal direct decomposition $TN=\mathcal{D}^{\theta_1} \oplus \mathcal{D}^{\theta_2}$.
\item The distributions $\mathcal{D}^{\theta_1}$ and $\mathcal{D}^{\theta_2}$ are slant distributions with slant angles $\theta_1$ and $\theta_2$.
\end{enumerate}
It is easy to see that if $\mathcal{D}^{\theta_1}=\{0\}$ (or $\mathcal{D}^{\theta_2}=\{0\}$), $ N$ becomes a slant submanifold with a slant angle $\theta_1$. If $\theta=\theta_1=\theta_2=\{0\}$, then $ N$ becomes a holomorphic submanifold. If $\theta=\theta_1=\theta_2=\frac{\pi}{2}$, then $ N$ becomes an anti-invariant submanifold. Moreover if $\theta_1=\frac{\pi}{2}$ and $\theta_2=0$, then $ N$ becomes a CR-submanifold. Furthermore, $\theta_1=\frac{\pi}{2}$ and $\theta_1=0$, then $ N$ becomes a hemi-slant submanifold and semi-slant submanifold, respectively. We denote bi-slant submanifolds by $v_7$.

Digraph $D_5 = (V,A)$ has seven vertices, $V(D_5) = \{v_1,v_2,v_3,v_4,v_5,v_6,v_7\}$, and 12 arcs, $A(D_4)=\{(v_2,v_1),(v_2,v_3),(v_4,v_1),(v_4,v_3),(v_5,v_2),(v_5,v_3),\\(v_5,v_4),(v_6,v_1),(v_6,v_2),(v_6,v_3),(v_7,v_5),(v_7,v_6)\}$ in Fig. \ref{fig6}. $D_5$ has the maximum length of three as the longest path. It has a vertex ($v_7$) which is not reachable. Using transitive closure, $D_5$ has five new direct connections such as $v_5 \to v_1$, $v_7 \to v_1$, $v_7 \to v_2$, $v_7 \to v_3$ and $v_7 \to v_4$. Topological sort of $D_5$ is $v_7-v_6-v_5-v_4-v_2-v_3-v_1$. $rad(D_5)=1$, the radius of $D_5$ is $v_2 \to v_1$. $diam(D_5)=2$, the diameter of $D_5$ is $v_5 \to v_2 \to v_1$. Also, in $D_5$, there are three center vertices as $v_2$, $v_4$ and $v_6$, and two peripheral vertices as $v_5$ and $v_7$.

	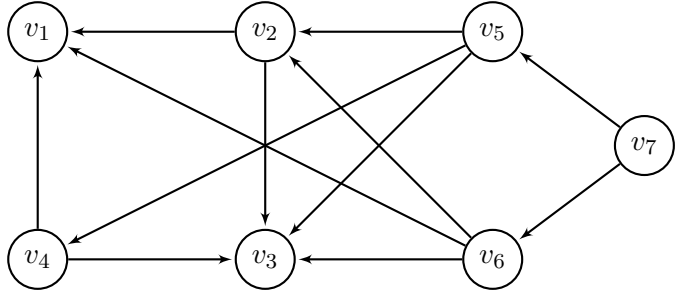
\begin{figure}
    \centering
\begin{tikzpicture}[->,>=stealth',shorten >=1pt,thick]
\tikzset{vertex/.style = {shape=circle,draw,minimum size=1.5em}}
\tikzset{edge/.style = {->,> = latex'}}
\node[vertex] (1) at  (0,0) {$v_1$};
\node[vertex] (2) at  (3,0) {$v_2$};
\node[vertex] (3) at  (3,-3) {$v_3$};
\node[vertex] (4) at  (0,-3) {$v_4$};
\node[vertex] (5) at  (6,0) {$v_5$};
\node[vertex] (6) at  (6,-3) {$v_6$};
\node[vertex] (7) at  (8,-1.5) {$v_7$};
\draw[edge] (2) to (1);
\draw[edge] (2) to (3);
\draw[edge] (4) to (1);
\draw[edge] (4) to (3);
\draw[edge] (5) to (2);
\draw[edge] (5) to (3);
\draw[edge] (5) to (4);
\draw[edge] (6) to (1);
\draw[edge] (6) to (2);
\draw[edge] (6) to (3);
\draw[edge] (7) to (5);
\draw[edge] (7) to (6);
\end{tikzpicture}
    \caption{Digraph $D_5$ built by submanifolds in $D_4$ and the bi-slant submanifold}
    \label{fig6}
    \end{figure}

\begin{theorem}
\label{teo35}
For the digraph $D_5$ created by adding the bi-slant submanifolds as vertex $v_7$ to the $D_4$,
\begin{enumerate}[i]
\item $D_5$ is a three-partite digraph.
\item The maximum matching is 3.
\item The independence number is 3.
\item The chromatic number is 3.
\item $D_5$ has a directed spanning tree.
\item The domination number is 3.
\end{enumerate}
\end{theorem}

\begin{proof}
\begin{enumerate}[i]
\item There is a partition $V_1=\{v_1,v_3\}$, $V_2=\{v_2, v_4, v_7\}$ and $V_3=\{v_5, v_6\}$ of $V(D_5)$ as three subsets, and these subsets support that $V(D_4) = \bigcup\limits_{i=1}^{3} V_{i}$ and $V_i \cap V_j = \O$ ($i,j=1,2,3$ and $i \neq j$). Then, $D_5$, containing the subsets, is actually a three-partite digraph.

\item $M=\{(v_2,v_1),(v_5,v_4),(v_6,v_3)\}$ is an arc subset in $D_5$, and $|M| = 3$. According to this, $M$, that includes no common vertices and no loops, is a matching. Since no other subset greater cardinality than $M$, $D_5$ has a maximum matching called $M$.

\item The subset $\widetilde{V}=\{v_2, v_4, v_7\}$ is an independent set having maximum cardinality. It also means that there is no direct relationship between any two elements, i.e. submanifolds, in $\widetilde{V}$. Then, the independence number of $D_5$ is three, because $|\widetilde{V}|=3$.

\item The minimum number of disjoint independent sets of $D_5$ is three: $V_1=\{v_1, v_3\}$, $V_2=\{v_2, v_4, v_7\}$ and $V_3=\{v_5, v_6\}$. According to that, three different colors are needed to coloring $D_5$ and that's why the chromatic number of $D_5$ is three.

\item $D_5$ has a directed spanning tree of length 6 and root at $v_7$ such as in Fig. \ref{fig7}. It also means that there is a transformation from submanifolds $v_7$ to all other submanifolds in $D_5$ at most two-step.

	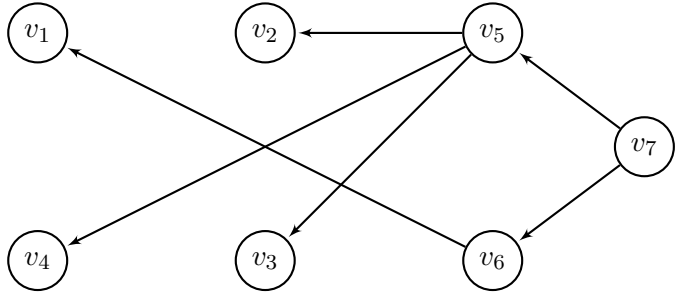
\begin{figure}[h]
    \centering
\begin{tikzpicture}[->,>=stealth',shorten >=1pt,thick]
\tikzset{vertex/.style = {shape=circle,draw,minimum size=1.5em}}
\tikzset{edge/.style = {->,> = latex'}}
\node[vertex] (1) at  (0,0) {$v_1$};
\node[vertex] (2) at  (3,0) {$v_2$};
\node[vertex] (3) at  (3,-3) {$v_3$};
\node[vertex] (4) at  (0,-3) {$v_4$};
\node[vertex] (5) at  (6,0) {$v_5$};
\node[vertex] (6) at  (6,-3) {$v_6$};
\node[vertex] (7) at  (8,-1.5) {$v_7$};
\draw[edge] (5) to (2);
\draw[edge] (5) to (3);
\draw[edge] (5) to (4);
\draw[edge] (6) to (1);
\draw[edge] (7) to (5);
\draw[edge] (7) to (6);
\end{tikzpicture}
    \caption{A directed spanning tree in $D_5$}
    \label{fig7}
	\end{figure}

\item There is a subset $\widetilde{V}=\{v_5, v_6, v_7\} \subset V(D_5)$. According to this subset, that having the minimum cardinality, and for each vertex $v \in \widetilde{V}$ and $u \in V(D_5)-\widetilde{V}$, $(v,u)$ is an arc in $D_5$. The domination number is three.
\end{enumerate}
\end{proof}

\begin{corollary}\label{cor35}
In the submanifold network represented by $D_5$ in Fig. \ref{fig6}, all other submanifolds can be derivated from bi-slant $(v_7)$ submanifold since $v_7$ is the root vertex of the directed spanning tree of $D_5$ and $N^{+}_{D_5}(v_7) = \{v_5, v_6\}$ in Fig. \ref{fig7}. Also, no submanifolds can be transformed to bi-slant because $N^{-}_{D_5}(v_7) = \emptyset$.
\end{corollary}

Digraph $D_6 = (V,A)$ has also seven vertices as well as $D_5$, $V(D_6) = \{v_1,v_2,v_3,v_4,v_5,v_6,v_7\}$, and 12 arcs, $A(D_6)=\{(v_2,v_1),(v_2,v_3),(v_4,v_1),(v_4,v_3),(v_5,v_1),(v_5,v_2),(v_5,v_3),(v_5,v_4),(v_6,v_1),\\(v_6,v_2),(v_6,v_3),(v_6,v_4),(v_7,v_5),(v_7,v_6)\}$ in Fig. \ref{fig8}. $D_6$ has the maximum length of three as the longest path. It has a vertex ($v_7$) which is not reachable. Using transitive closure, $D_6$ has four new direct connections such as $v_7 \to v_1$, $v_7 \to v_2$, $v_7 \to v_3$ and $v_7 \to v_4$. Topological sort of $D_6$ is $v_7-v_6-v_5-v_4-v_2-v_3-v_1$. $rad(D_6)=1$, the radius of $D_6$ is $v_2 \to v_1$. $diam(D_6)=2$, the diameter of $D_6$ is $v_7 \to v_5 \to v_1$. Also, in $D_6$, there are four center vertices as $v_2$, $v_4$, $v_5$ and $v_6$, and one peripheral vertex as $v_7$.

	\begin{figure}
    \centering
\begin{tikzpicture}[->,>=stealth',shorten >=1pt,thick]
\tikzset{vertex/.style = {shape=circle,draw,minimum size=1.5em}}
\tikzset{edge/.style = {->,> = latex'}}
\node[vertex] (1) at  (0,0) {$v_1$};
\node[vertex] (2) at  (3,-0.5) {$v_2$};
\node[vertex] (3) at  (3,-2.5) {$v_3$};
\node[vertex] (4) at  (0,-3) {$v_4$};
\node[vertex] (5) at  (6,0) {$v_5$};
\node[vertex] (6) at  (6,-3) {$v_6$};
\node[vertex] (7) at  (8,-1.5) {$v_7$};
\draw[edge] (2) to (1.345);
\draw[edge] (2) to (3);
\draw[edge] (4) to (1);
\draw[edge] (4.015) to (3);
\draw[edge] (5) to (1);
\draw[edge] (5.195) to (2);
\draw[edge] (5) to (3);
\draw[edge] (5) to (4);
\draw[edge] (6) to (1);
\draw[edge] (6) to (2);
\draw[edge] (6.165) to (3);
\draw[edge] (6) to (4);
\draw[edge] (7) to (5);
\draw[edge] (7) to (6);
\end{tikzpicture}
    \caption{Digraph $D_6$ built by $D_5$ with arcs $(v_5,v_1)$ and $(v_6,v_4)$}
    \label{fig8}
    \end{figure}
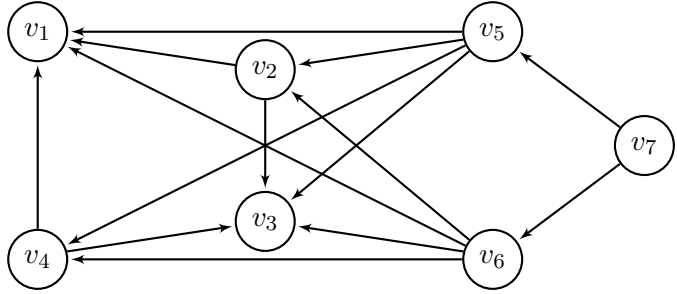

\begin{theorem}
\label{teo36}
For the digraph $D_6$ created by adding two more relations from semi-slant to holomorphic and from hemi-slant to slant as arcs to the $D_5$,
\begin{enumerate}[i]
\item $D_6$ is a three-partite digraph.
\item The maximum matching is 3.
\item The independence number is 3.
\item The chromatic number is 3.
\item $D_6$ contains a directed spanning tree.
\item The domination number is 2.
\end{enumerate}
\end{theorem}

\begin{proof}
\begin{enumerate}[i]
\item See Teorem \ref{teo35}-i.
\item See Teorem \ref{teo35}-ii.
\item See Teorem \ref{teo35}-iii.
\item See Teorem \ref{teo35}-iv.
\item $D_6$ has a directed spanning tree having the same structure as in Fig. \ref{fig7}. See Theorem \ref{teo35}-v.
\item There is a subset $\widetilde{V}=\{v_5, v_7\} \subset V(D_6)$. According to that, the subset has the minimum cardinality while dominating all other vertices, and for each vertex $v \in \widetilde{V}$ and $u \in V(D_6)-\widetilde{V}$, $(v,u)$ is an arc in $D_6$. The domination number is two.
\end{enumerate}
\end{proof}

\begin{corollary}\label{cor36}
In the most comprehensive submanifold network represented by $D_6$ in Fig. \ref{fig7}, just two submanifolds, holomorphic $(v_1)$ and anti-invariant $(v_3)$, are not generative since $d^{+}(v_1) = d^{+}(v_3) = 0$. Besides, bi-slant $(v_7)$ is the most productive submanifold owing to transforming to all other submanifolds.
\end{corollary}

Using the seven submanifolds, named as holomorphic, CR, anti-invariant, slant hemi-slant, semi-slant and bi-slant, it is constructed six digraphs, called $D_1, D_2, D_3, D_4, D_5$ and $D_6$, whose vertices are submanifolds and arcs are connections among submanifolds from one to another.

\begin{theorem}
\label{teo37}
Let $D$ be a digraph which indicates digraphs from $\{D_1, D_2, D_3, D_4, D_5, D_6\}$. $D$ provides the following properties:

\begin{enumerate}[i]
\item Simple directed graph.
\item Directed acyclic graph.
\item Weakly connected.
\end{enumerate}
\end{theorem}

\begin{proof}
\begin{enumerate}[i]
\item In digraph $D$, there is no more than one relationship between any two submanifolds and no transformations from a submanifold to itself. According to that, $D$ is a simple directed graph.

\item Given a transition list among submanifolds such as $v_1v_2...v_k$, meaning that $v_1$ is the source submanifold and $v_k$ is the sink submanifold. Because $D$ doesn't have any transition list including the same submanifold is both source and also sink, $D$ is acyclic. That's why $D$ is a directed acyclic digraph.

\item $D$ has one pair of submanifolds as a relation at least that they can not mutually be transformed from one to another submanifold. Hence, $D$ is not strongly connected. However, when $D$, that considered as without direction of transformations, is connected, named connectedness of underlying graph because there are no isolated submanifolds. For this reason, $D$ is weakly connected.
\end{enumerate}
\end{proof}

\begin{corollary}
\label{cor38}
Among all digraphs $D_1$, $D_2$, $D_3$, $D_4$, $D_5$, and $D_6$, the digraph $D_6$ has 
\begin{enumerate}
\item[•] the maximum vertex-integrity, and
\item[•] the maximum edge-integrity
\end{enumerate}
as well as the maximum size by Proposition \ref{prop.integrity}.
\end{corollary}

\section{Conclusion}
Manifold learning plays an important role in analyzing data lying on a non-linear space as a part of machine learning. Moreover, the geometric deep learning yields using the concepts of manifolds and graphs together in building convolutional deep structures.  In this paper, using holomorphic submanifolds, anti-invariant submanifolds, CR-submanifolds, slant submanifolds, semi-slant submanifolds, hemi-slant submanifolds and bi-slant submanifolds in almost Hermitian manifolds, it is given relations among them, six different digraphs are created as a network of these submanifolds, and main properties of them are first examined in terms of digraphs. We note that there is a much wider class that includes slant submanifolds. This class was first defined in \cite{Etayo} by Etayo as quasi-slant submanifolds. Later, these submanifolds were called pointwise slant submanifolds in \cite{Chen-Garay} by Chen and Garay. Although we have excluded such submanifolds in this article, our next research will be to examine the connections between these submanifolds and graph theory.

\bibliographystyle{fundam}

\end{document}